\documentclass[reqno]{amsart}
\usepackage[colorlinks,linkcolor=black,anchorcolor=black,citecolor=black,hyperindex=black,CJKbookmarks=black]{hyperref}

\newtheorem{theorem}{Theorem}[section]
\newtheorem{lemma}[theorem]{Lemma}
\newtheorem{proposition}{Proposition}[section]

\theoremstyle{definition}
\newtheorem{definition}[theorem]{Definition}

\theoremstyle{remark}
\newtheorem{remark}[theorem]{Remark}
\newtheorem{corollary}[theorem]{Corollary}

\numberwithin{equation}{section}

\begin{document}

\title[Beta-expansion and continued fraction expansion]{Beta-expansion and continued fraction expansion of real numbers}

\author {Lulu Fang, Min Wu and Bing Li$^{*}$}
\address{Department of Mathematics, South China University of Technology, Guangzhou 510640, P.R. China}
\email{fanglulu1230@163.com, wumin@scut.edu.cn and scbingli@scut.edu.cn}


\thanks {* Corresponding author}

\begin{abstract}
 Let $\beta > 1$ be a real number and $x \in [0,1)$ be an irrational number. We denote by $k_n(x)$ the exact number of partial quotients in the continued fraction expansion of $x$ given by the first $n$ digits in the $\beta$-expansion of $x$ ($n \in \mathbb{N}$). It is known that $k_n(x)/n$ converges to $(6\log2\log\beta)/\pi^2$ almost everywhere in the sense of Lebesgue measure. In this paper, we improve this result by proving that the Lebesgue measure of the set of $x \in [0,1)$ for which $k_n(x)/n$ deviates away from $(6\log2\log\beta)/\pi^2$ decays to 0 exponentially as $n$ tends to $\infty$, which generalizes the result of Faivre \cite{lesFai97} from $\beta = 10$ to any $\beta >1$. Moreover, we also discuss which of the $\beta$-expansion and continued fraction expansion yields the better approximations of real numbers.

\end{abstract}

\maketitle

\section{Introduction}

Let $\beta > 1$ be a real number and $T_\beta: [0,1) \longrightarrow [0,1)$ be the \emph{$\beta$-transformation} defined as
\begin{equation*}\label{T-beta}
T_\beta(x)= \beta x - [\beta x],
\end{equation*}
where $[x]$ denotes the greatest integer not exceeding $x$. Then every $x \in [0,1)$ can be uniquely expanded into a finite or infinite series, i.e.,
\begin{equation}\label{beta expansion}
x = \frac{\varepsilon_1(x)}{\beta} + \frac{\varepsilon_2(x)}{\beta^2} + \cdots + \frac{\varepsilon_n(x)}{\beta^n} + \cdots,
\end{equation}
where $\varepsilon_1(x) = [\beta x]$ and $\varepsilon_{n+1}(x) = \varepsilon_1(T_\beta^nx)$ for all $n \geq 1$. We call the representation (\ref{beta expansion}) the \emph{$\beta$-expansion} of $x$ denoted by $(\varepsilon_1(x),\varepsilon_2(x),\cdots, \varepsilon_n(x),\cdots)$ and $\varepsilon_n(x), n \geq 1$ the \emph{digits} of the $\beta$-expansion of $x$. Such an expansion was first introduced by  R\'{e}nyi \cite{lesRen57}, who proved that there exists a unique $T_\beta$-invariant measure $\mu$ equivalent to the Lebesgue measure $\lambda$. More precisely,
\begin{equation}\label{beta measure}
C^{-1} \lambda(A) \leq \mu(A) \leq C \lambda(A)
\end{equation}
for any Borel set $A \subseteq [0,1)$, where $C>1$ is a constant only depending on $\beta$. Furthermore, Gel'fond \cite{lesGel59} and Parry \cite{lesPar60} independently found the density formula for this invariant measure with respect to (w.r.t.)~the Lebesgue measure. Philipp \cite{lesPhi67} showed that the dynamical system $([0,1), \mathcal{B}, T_{\beta}, \mu)$ is an exponentially strong mixing measure-preserving system,
where $\mathcal{B}$ is the Borel $\sigma$-algebra on $[0,1)$. Later, Hofbauer \cite{lesHof78} proved that $\mu$ is the unique measure of maximal entropy for $T_{\beta}$. Some arithmetic and metric properties of $\beta$-expansion were studied in the literature, such as \cite{lesAB07, lesBer12, lesBla89, lesD.K02, lesF.W12, lesKL15, lesSch97, lesSch80} and the references therein.

Now we turn to introducing continued fractions. Let $T: [0,1) \longrightarrow [0,1)$ be the \emph{Gauss transformation} given by
\begin{equation*}
T(0):=0\ \ \  \text{and}\ \ \ T(x):= 1/x- [1/x]\ \ \  \text{if}\ \ \  x \in (0,1).
\end{equation*}
Then any irrational number $x \in [0,1)$ can be written as
\begin{equation}\label{continued fraction expansion}
x = \dfrac{1}{a_1(x) +\dfrac{1}{a_2(x) + \ddots +\dfrac{1}{a_n(x)+ \ddots}}},
\end{equation}
where $a_1(x) = [1/x]$ and $a_{n +1}(x) = a_1(T^nx)$ for all $n \geq 1$. The form (\ref{continued fraction expansion}) is said to be the \emph{continued fraction expansion} of $x$ and $a_n(x), n \geq 1$  are called the \emph{partial quotients} of the continued fraction expansion of $x$.  Sometimes we write the form (\ref{continued fraction expansion}) as $[a_1(x), a_2(x), \cdots, a_n(x), \cdots]$. For any $n \geq 1$, we denote by $\frac{p_n(x)}{q_n(x)}:= [a_1(x), a_2(x), \cdots, a_n(x)]$ the $n$-th \emph{convergent} of the continued fraction expansion of $x$, where $p_n(x)$ and $q_n(x)$ are relatively prime. Clearly these convergents are rational numbers and $p_n(x)/q_n(x) \to x$ as $n \to \infty$ for all $x \in [0,1)$. For more details about continued fractions, we refer the reader to a monograph of Khintchine \cite{lesKhi64}.

A natural question is whether there exists some relationship between different expansions of some real number, for instance, its $\beta$-expansion and continued fraction expansion. For any irrational number $x \in [0,1)$ and $n \geq 1$, we denote by $k_n(x)$ the exact number of partial quotients in the continued fraction expansion of $x$ given by the first $n$ digits in the $\beta$-expansion of $x$. In other words,
\[
k_n(x) = \sup\left\{m \geq 0: J(\varepsilon_1(x),\cdots, \varepsilon_n(x)) \subseteq I(a_1(x), \cdots, a_m(x))\right\},
\]
where $J(\varepsilon_1(x),\cdots, \varepsilon_n(x))$ and $I(a_1(x), \cdots, a_m(x))$ are called the \emph{cylinders} of $\beta$-expansion and continued fraction expansion respectively (see Section 2 for the definition of the cylinder). It is easy to check that
\begin{equation*}\label{increasing sequence}
0 \leq k_1(x) \leq k_2(x) \leq \cdots\ \ \text{and}\ \ \lim\limits_{n \to \infty}k_n(x) = \infty.
\end{equation*}
The quantity $k_n(x)$ was first introduced by Lochs \cite{lesLoc64} for $\beta = 10$ and has been extensively studied by many mathematicians, see \cite{lesB.I08, lesFai97, lesFai98, lesFWL, lesL.W08, lesWu06, lesWu08}.
Applying the result of Dajani and Fieldsteel \cite{lesD.F01} (Theorem 5) to $\beta$-expansion and continued fraction expansion, Li and Wu \cite{lesL.W08} obtained a metric result of $\{k_n, n \geq 1\}$, that is, for $\lambda$-almost all $x \in [0,1)$,
\begin{equation}\label{theorem eq}
\lim\limits_{n \to \infty}\frac{k_n(x)}{n} = \frac{6\log2\log\beta}{\pi^2}.
\end{equation}
The formula (\ref{theorem eq}) has been stated for $\beta=10$ by a pioneering result of Lochs \cite{lesLoc64}. Barreira and Iommi \cite{lesB.I08} proved that the irregular set of points $x \in [0,1)$ for which the limit in (\ref{theorem eq}) does not exist has full Hausdorff dimension. Li and Wu \cite{lesL.W08} gave some asymptotic results of $k_n(x)/n$ for any irrational $x \in [0,1)$ not just a kind of almost all result (see also Wu \cite{lesWu06}). For the special case $\beta = 10$, some limit theorems of $\{k_n, n \geq 1\}$ were studied in the earlier literature. For instance, using Ruelle-Mayer operator, Faivre \cite{lesFai97} showed an error term: for positive $\varepsilon$ the Lebesgue measure of the set of $x$'s for which $k_n(x)/n$ is more than $\varepsilon$ away from $(6\log2\log10)/\pi^2$
tends geometrically to zero. Later, he also proved a central limit theorem for the sequence $\{k_n, n \geq 1\}$ in \cite{lesFai98}. The law of the iterated logarithm for the sequence $\{k_n, n \geq 1\}$  was established by Wu \cite{lesWu08}.

We wonder if the similar limit theorems of the sequence $\{k_n, n \geq 1\}$ are still valid for general $\beta >1$. It is worth pointing out that the lengths of cylinders play an important role in the study of $\beta$-expansion (see \cite{lesB.W14, lesF.W12}). The methods of Faivre (see \cite{lesFai97, lesFai98}) and Wu (see \cite{lesWu06, lesWu08}) rely heavily on the length of a cylinder for $\beta = 10$.
Indeed, the $n$-th cylinder is a regular interval and its length equals always to $10^{-n}$ for $\beta = 10$. For the general case $\beta >1$, it is well-known that the $n$-th cylinder is a left-closed and right-open interval and its length has an absolute upper bound $\beta^{-n}$. Fan and Wang \cite{lesF.W12} obtained that the growth of the lengths of cylinders is multifractal and that the multifractal spectrum depends on $\beta$. However, for some ``bad'' $\beta >1$, the  $n$-th cylinder is irregular and there is no nontrivial absolute lower bound for its length, which can be much smaller than $\beta^{-n}$. This is the main difficulty we met. In fact, the authors have established a lower bound (not necessarily absolute) of the length of a cylinder (see \cite[Proposition 2.3]{lesFWL}) and obtained the central limit theorem and law of the iterated logarithm of $\{k_n, n \geq 1\}$ for any $\beta >1$ in \cite{lesFWL}.
In the present paper, we make use of this lower bound to extend the result of Faivre \cite{lesFai97} from $\beta =10$ to any $\beta >1$, which indicates that the Lebesgue measure of the set of $x \in [0,1)$ for which $k_n(x)/n$ deviates away from $(6\log2\log\beta)/\pi^2$ tends to 0 exponentially as $n$ goes to $\infty$.

\begin{theorem}\label{error term}
Let $\beta >1$ be a real number. For any $\varepsilon > 0$, there exist two positive constants $A $ and $\alpha$ (both depending on $\beta$ and $\varepsilon$) such that for all $n \geq 1$,
\[
 \lambda\left\{x \in [0,1): \left|\frac{k_n}{n}- \frac{6\log2\log\beta}{\pi^2}\right| \geq \varepsilon\right\} \leq Ae^{-\alpha n}.
\]
\end{theorem}

\begin{remark}
The above theorem immediately yields that for all $\varepsilon >0$,
\[
\sum\limits_{n=1}^\infty  \lambda\left\{x \in [0,1): \left|\frac{k_n}{n}- \frac{6\log2\log\beta}{\pi^2}\right| \geq \varepsilon\right\}  < + \infty.
\]
That is, $k_n(x)/n $ converges completely to $(6\log2\log\beta)/\pi^2$ (the definition of the complete convergence see \cite{lesHR47}). This kind of convergence is stronger than almost everywhere convergence and sometimes more convenient to establish.
By Borel-Cantelli lemma, we easily obtain the limit (\ref{theorem eq}) for $\lambda$-almost all $x \in [0,1)$.
\end{remark}

For any $x \in [0,1)$, we denote the partial sums of the form (\ref{beta expansion}) by
\[
x_n = \frac{\varepsilon_1(x)}{\beta} + \frac{\varepsilon_2(x)}{\beta^2} + \cdots + \frac{\varepsilon_n(x)}{\beta^n}
\]
and call them the \emph{convergents} of the $\beta$-expansion of $x$. It is clear that the sequence $\{x_n, n \geq 1\}$ converges to $x$ as $n \to \infty$ for all $x \in [0,1)$. If $ x - x_n > \left|x-{p_n}/{q_n}\right|$, we say that the approximation of $x$ by ${p_n}/{q_n}$ is better than the approximation by $x_n$ ($n \in \mathbb{N}$). The formula (\ref{theorem eq}) implies that for $\lambda$-almost all $x \in [0,1)$, the larger $\beta$ is (that is, the more symbols we use to code number $x$), the more information about the continued fraction expansion we can obtained from its $\beta$-expansion. In other words, for sufficient large $\beta>1$,
the approximation of a real number by $\beta$-expansion is better than the approximation by continued fractions. More precisely, we show that the Lebesgue measure of the set for which the first $n$ partial quotients of continued fraction expansion provide a better approximation for $x$ than the first $n$ digits of $\beta$-expansion decreases to 0 geometrically as $n$ tends to $\infty$ if $\log\beta > {\pi^2}/{(6\log2)}$ and the case is converse when $\log\beta < {\pi^2}/{(6\log2)}$. Besides, we can also see that the decay rates are related to the multifractal analysis for the Lyapunov exponent of the Gauss transformation (see Remark \ref{3.6}).

\begin{theorem}\label{better approximation}
Let $\beta >1$ be a real number.\\
(i) If $\log\beta > {\pi^2}/{(6\log2)}$, then there exist two constants $B_1>0$ and $\gamma_1>0$ (both only depending on $\beta$) such that for all $n \geq 1$,
\[
\lambda\left\{x \in [0,1):\left|x - \frac{p_n}{q_n}\right| \leq x - x_n\right\} \leq B_1e^{-\gamma_1 n}.
\]
(ii) If $\log\beta < {\pi^2}/{(6\log2)}$, then there exist two constants $B_2>0$ and $\gamma_2>0$ (both only depending on $\beta$) such that for all $n \geq 1$,
\[
 \lambda\left\{x \in [0,1): x - x_n \leq \left|x - \frac{p_n}{q_n}\right|\right\} \leq B_2e^{-\gamma_2 n}.
\]
\end{theorem}

By Borel-Cantelli lemma, we immediately obtain the following corollary.

\begin{corollary}
Let $\beta >1$ be a real number.\\
(i) If $\log\beta > {\pi^2}/{(6\log2)}$, then for $\lambda$-almost all $x \in [0,1)$, there exists positive integer $N_1$ (depending on $x$) such that for all $n \geq N_1$,
\[
\left|x - \frac{p_n(x)}{q_n(x)}\right| >  x - x_n.
\]
(ii) If $\log\beta < {\pi^2}/{(6\log2)}$, then for $\lambda$-almost all $x \in [0,1)$, there exists positive integer $N_2$ (depending on $x$) such that for all $n \geq N_2$,
\[
\left|x - \frac{p_n(x)}{q_n(x)}\right| <  x - x_n.
\]
\end{corollary}

\begin{remark}
Since $\log 10 < {\pi^2}/{(6\log2)}$, we know that the approximation of some real number by decimal expansion (i.e., $\beta =10$) is not better than the approximation by continued fraction expansion in the view of almost everywhere. This  result was obtained by Faivre \cite{lesFai97} in 1997. 
However, for the critical value $\beta = \exp({\pi^2}/{(6\log2)}) \approx 10.731$, our methods are invalid and we do not know which is the better approximation by $\beta$-expansion and continued fraction expansion.
\end{remark}

\section{Preliminary}
This section is devoted to recalling some definitions and basic properties of the $\beta$-expansion and continued fraction expansion.

\subsection{$\beta$-expansions}

We first state some notions and basic properties of $\beta$-expansion.

\begin{definition}
An $n$-block $(\varepsilon_1 ,\varepsilon_2, \cdots, \varepsilon_n)$ is said to be admissible for $\beta$-expansions if there exists $x \in [0,1)$ such that $\varepsilon_i(x) = \varepsilon_i$ for all $1 \leq i \leq n$. An infinite sequence $(\varepsilon_1 ,\varepsilon_2, \cdots, \varepsilon_n, \cdots)$ is admissible if $(\varepsilon_1 ,\varepsilon_2, \cdots, \varepsilon_n)$ is admissible for all $n \geq 1$.
\end{definition}

We denote by $\Sigma_\beta^n$ the collection of all admissible sequences of length $n$. The following result of R\'{e}nyi \cite{lesRen57} implies that the dynamical system ([0,1), $T_\beta$) admits $\log \beta$ as its topological entropy.

\begin{proposition}[\cite{lesRen57}]\label{Renyi}
Let $\beta >1$. For any $n \geq 1$,
\begin{equation*}
\beta^n \leq \sharp \Sigma_\beta^n \leq \beta^{n+1}/(\beta -1),
\end{equation*}
where $\sharp$ denotes the number of elements of a finite set.
\end{proposition}

\begin{definition}
Let $(\varepsilon_1 ,\varepsilon_2, \cdots, \varepsilon_n)\in \Sigma_\beta^n$. We define
\[
J(\varepsilon_1 ,\varepsilon_2, \cdots, \varepsilon_n) = \{x \in [0,1):\varepsilon_i(x)=\varepsilon_i \ \text{for all} \ 1 \leq i \leq n\}
\]
and call it the $n$-th cylinder of $\beta$-expansion. In other words, it is the set of points beginning with $(\varepsilon_1 ,\varepsilon_2, \cdots, \varepsilon_n)$ in their $\beta$-expansions. For any real number $x \in [0,1)$, $J(\varepsilon_1(x) ,\varepsilon_2(x), \cdots, \varepsilon_n(x))$ is said to be the $n$-th cylinder containing $x$.
\end{definition}

\begin{remark}\label{beta cylinder}
The cylinder $J(\varepsilon_1 ,\varepsilon_2, \cdots, \varepsilon_n)$ is a left-closed and right-open interval with left endpoint
\begin{equation*}
\frac{\varepsilon_1}{\beta} + \frac{\varepsilon_2}{\beta^2} + \cdots + \frac{\varepsilon_n}{\beta^n}.
\end{equation*}
Moreover, the length of $J(\varepsilon_1 ,\varepsilon_2, \cdots, \varepsilon_n)$ satisfies
$|J(\varepsilon_1 ,\varepsilon_2, \cdots, \varepsilon_n)| \leq 1/\beta^n$.
\end{remark}

For any $x \in [0,1)$ and $n \geq 1$, we
assume that $(\varepsilon_1(x), \varepsilon_2(x), \cdots, \varepsilon_n(x), \cdots)$ is the $\beta$-expansion of $x$ and define
\begin{equation}\label{l-n}
l_n(x) = \sup\left\{k \geq 0: \varepsilon_{n +j}(x) =0 \ \text{for all}\ 1 \leq j \leq k\right\}.
\end{equation}
That is, the length of the longest string of zeros just after the $n$-th digit in the $\beta$-expansion of $x$. The following proposition gives a lower bound, which is not absolute and related to $l_n(x)$.

\begin{proposition}[\cite{lesFWL}]\label{zhongyao1}
Let $\beta >1$. Then for any $x \in [0,1)$ and $n \geq 1$,
\begin{equation*}
\frac{1}{\beta^{n+l_n(x)+1}}\leq |J(\varepsilon_1(x) ,\varepsilon_2(x), \cdots, \varepsilon_n(x))| \leq \frac{1}{\beta^n},
\end{equation*}
where $l_n(x)$ is defined as (\ref{l-n}).
\end{proposition}

\begin{proof}
For any $x \in [0,1)$ and $n \geq 1$, we know that $J(\varepsilon_1(x) ,\varepsilon_2(x), \cdots, \varepsilon_n(x))$ is a left-closed and right-open interval with left endpoint
\begin{equation*}
\omega_n(x):= \frac{\varepsilon_1(x)}{\beta} + \frac{\varepsilon_2(x)}{\beta^2} + \cdots + \frac{\varepsilon_n(x)}{\beta^n}
\end{equation*}
and its length satisfies $|J(\varepsilon_1(x) ,\varepsilon_2(x), \cdots, \varepsilon_n(x))| \leq 1/\beta^n.$

Since $x \in J(\varepsilon_1(x) ,\varepsilon_2(x), \cdots, \varepsilon_n(x))$, we have that
\begin{equation*}
\frac{1}{\beta^{n+l_n(x)+1}} \leq x - \omega_n(x) \leq |J(\varepsilon_1(x) ,\varepsilon_2(x), \cdots, \varepsilon_n(x))|,
\end{equation*}
where the first inequality follows from $\varepsilon_{n+1}(x) = \cdots = \varepsilon_{n+l_n(x)}(x) = 0$ and $\varepsilon_{n+l_n(x)+1}(x) \geq 1$ by the definition of $l_n(x)$ in (\ref{l-n}). This completes the proof.
\end{proof}

\subsection{Continued fractions}

Let us now introduce some elementary properties of continued fractions. For any irrational number $x \in [0,1)$ and $n \geq 1$, with the conventions $p_{-1}=1$, $q_{-1}=0$, $p_0=0$, $q_0=1$, the quantities $p_n$ and $q_n$ satisfy the following recursive formula:
\begin{equation}\label{recursive}
p_n(x) = a_n(x) p_{n-1}(x) + p_{n-2}(x)\ \ \ \text{and}\ \ \ q_n(x) = a_n(x) q_{n-1}(x) + q_{n-2}(x).
\end{equation}
By the above recursive formula of $p_n$ and $q_n$, we can easily obtain that
\begin{equation}\label{Diophantine}
\frac{1}{2q_{n+1}^2(x)} \leq \left|x - \frac{p_n(x)}{q_n(x)}\right| \leq \frac{1}{q_n^2(x)}.
\end{equation}
This is to say that the speed of $p_n(x)/q_n(x)$ approximating to $x$ is dominated by $q_n^{-2}(x)$. So the denominator of the $n$-th convergent $q_n(x)$ plays an important role in the problem of Diophantine approximation.

\begin{definition}\label{cylinder}
For any $n \geq 1$ and $a_1, a_2, \cdots, a_n \in \mathbb{N}$, we call
\begin{equation*}
I(a_1, \cdots, a_n):= \left\{x \in [0,1): a_1(x)=a_1, \cdots, a_n(x)=a_n\right\}
\end{equation*}
the $n$-th cylinder of continued fraction expansion. For any real number $x \in [0,1)$, $I(a_1(x) ,a_2(x), \cdots, a_n(x))$ is said to be the $n$-th cylinder containing $x$.
\end{definition}

The following proposition is about the structure and length of cylinders.

\begin{proposition}[\cite{lesD.K02}]
Let $a_1, a_2, \cdots, a_n \in \mathbb{N}$. Then $I(a_1, \cdots, a_n)$ is an interval with two endpoints
\[
\frac{p_n}{q_n}\ \ \ \ \ \text{and}\ \ \ \ \  \frac{p_n+p_{n-1}}{q_n+q_{n-1}}
\]
and the length of $I(a_1, \cdots, a_n)$ satisfies
\begin{equation}\label{cylinder inequality}
\frac{1}{2q_n^2}\leq |I(a_1, \cdots, a_n)| = \frac{1}{q_n(q_n+q_{n-1})} \leq \frac{1}{q_n^2},
\end{equation}
where $p_n$ and $q_n$ satisfy the recursive formula (\ref{recursive}).
\end{proposition}

\section{Proofs of theorems}

In the following, we denote by $\mathbb{I}$ the set of all irrational numbers in $[0,1)$ and use the notation $\mathrm{E}(\xi)$ to denote the expectation of a random variable $\xi$ w.r.t. the Lebesgue measure $\lambda$. For any $\theta > 1/2$, we define the so called \emph{Diophantine pressure function} (see Kesseb\"{o}hmer and Stratmann \cite{lesKS07}) as
\[
\mathrm{P}(\theta): = \lim_{n \to \infty}\frac{1}{n} \log \sum_{a_1,\cdots,a_n} q_n^{-2\theta}([a_1,\cdots,a_n]).
\]
Kesseb\"{o}hmer and Stratmann \cite{lesKS07} proved that the Diophantine pressure function has a singularity at $1/2$ and is decreasing, convex and real-analytic on $(1/2,+\infty)$ satisfying
\begin{equation}\label{ele}
\mathrm{P}(1)=0\ \ \ \ \text{and}\ \ \ \  \mathrm{P}^\prime(1)=-\pi^2/(6\log 2).
\end{equation}
Furthermore, they also studied the multifractal analysis for the Lyapunov exponent of the Gauss transformation $T$ by using this function (see also Pollicott and Weiss \cite{lesPW99} and Fan et al.~\cite{lesF.L.W.W} and \cite{lesF.L.W.W13}). More detailed analysis of this function can be founded in Mayer \cite{lesMay90}. The following lemma establishes the relation between this function and the growth of the expectation of $q_n$, which plays an important role in our proofs.

\begin{lemma}\label{proofs}
For any $\theta <1/2$,
\[
\mathrm{P}(1-\theta) = \lim_{n \to \infty}\frac{1}{n} \log \mathrm{E}\left(q_n^{2\theta}\right).
\]
\end{lemma}

\begin{proof}
By the definition of expectation, we know that
\begin{equation}\label{qiwang}
E\left(q_n^{2\theta}\right) = \sum_{a_1,\cdots,a_n} q_n^{2\theta}([a_1,\cdots,a_n])\cdot \lambda\left(I(a_1,\cdots,a_n)\right),
\end{equation}
where $a_1,\cdots,a_n$ run over all the positive integers. In view of (\ref{cylinder inequality}), we have that
\begin{equation*}
\frac{1}{2q_n^2([a_1,\cdots,a_n])}\leq \lambda\left(I(a_1,\cdots,a_n)\right) \leq \frac{1}{q_n^2([a_1,\cdots,a_n])}.
\end{equation*}
Combing this with (\ref{qiwang}), we deduce that
\[
\frac{1}{2}\cdot \sum_{a_1,\cdots,a_n} q_n^{-2(1-\theta)}([a_1,\cdots,a_n]) \leq E\left(q_n^\theta\right) \leq \sum_{a_1,\cdots,a_n} q_n^{-2(1-\theta)}([a_1,\cdots,a_n])
\]
and hence that
\[
\mathrm{P}(1-\theta) = \lim_{n \to \infty}\frac{1}{n} \log \mathrm{E}\left(q_n^{2\theta}\right) \ \ \ \text{for any $\theta <1/2$}.
\]
\end{proof}

\subsection{Proof of Theorem \ref{error term}}
Recall that
\[
k_n(x) = \sup\left\{m \geq 0: J(\varepsilon_1(x),\cdots, \varepsilon_n(x)) \subseteq I(a_1(x), \cdots, a_m(x))\right\},
\]
we have the following lemma.

\begin{lemma}\label{jihe}
Let $m \geq 1$ be an integer. Then
\[
\left\{x \in \mathbb{I} : k_n(x) \geq m\right\} = \left\{x \in \mathbb{I}: J(\varepsilon_1(x), \cdots, \varepsilon_n(x)) \subseteq I(a_1(x), \cdots, a_m(x))\right\}.
\]
\end{lemma}

To prove Theorem \ref{error term}, we will show a stronger result.

\begin{proposition}\label{first}
Let $a = (6\log2\log\beta)/\pi^2$. Then for any $\varepsilon >0$,
\begin{equation}\label{he}
\limsup_{n \to \infty} \frac{1}{n}\log \lambda \left\{x \in \mathbb{I}:\frac{k_n(x)}{n} \geq a + \varepsilon\right\} \leq \theta_1(\varepsilon)
\end{equation}
with
\[
\theta_1(\varepsilon) = \inf_{t>0} \left\{\frac{1}{t+1}\Big(t\log\beta + (a + \varepsilon)\mathrm{P}(t+ 1)\Big)\right\}< 0
\]
and for any $0<\varepsilon<a$,
\begin{equation}\label{cha}
\limsup_{n \to \infty} \frac{1}{n}\log \lambda \left\{x \in \mathbb{I}:\frac{k_n(x)}{n} \leq a - \varepsilon\right\} \leq \theta_2(\varepsilon)
\end{equation}
with
\[
\theta_2(\varepsilon) = \inf_{0 <t< 1/2} \big\{-t\log\beta + (a - \varepsilon)\mathrm{P}(1-t )\big\}< 0.
\]
\end{proposition}

\begin{remark}\label{add remarks}
By the domain of the function $\mathrm{P (\cdot)}$, we first remark that the quantities $\theta_1(\varepsilon)$ and $\theta_2(\varepsilon)$ can be rewritten as
\[
\theta_1(\varepsilon) = \inf_{t > -1/2} \left\{\frac{1}{t+1}\Big(t\log\beta + (a + \varepsilon)\mathrm{P}(t+ 1)\Big)\right\}
\]
and
\[
\theta_2(\varepsilon) = \inf_{t <1/2} \big\{-t\log\beta + (a - \varepsilon)\mathrm{P}(1-t)\big\}.
\]
In fact, for any $\varepsilon>0$, we define $f(t) = t\log\beta + (a + \varepsilon)\mathrm{P}(t+1)$ for all $-1/2 <t \leq 0$ and $g(t) = -t\log\beta + (a - \varepsilon)\mathrm{P}(1-t)$ for all $t\leq 0$.
Moreover, we actually obtained that $f(t) \geq 0$ and $g(t) \geq 0$. Since $\mathrm{P(\cdot)}$ is convex and real-analytic on $(1/2,+\infty)$, we know that $P^\prime(t+1) \leq P^\prime(1)$ for any $-1/2<t \leq 0$ and $P^\prime(1-t) \geq P^\prime(1)$ for any $t \leq 0$.
It follows from (\ref{ele}) that
\[
f^\prime(t) = \log\beta + (a + \varepsilon)P^\prime(t+1) \leq \log\beta + (a + \varepsilon)P^\prime(1) = -\varepsilon\pi^2/(6\log 2)<0
\]
for any $-1/2<t \leq 0$ and
\[
g^\prime(t) = -\log\beta - (a - \varepsilon)P^\prime(1-t) \leq -\log\beta - (a - \varepsilon)P^\prime(1) = -\varepsilon\pi^2/(6\log 2)<0
\]
for any $t\leq0$.
Therefore, $f$ is decreasing on $(-1/2,0]$ and $g$ is decreasing on $(-\infty,0]$. In view of (\ref{ele}), we obtain that $f(t) \geq f(0) =0$ when $ -1/2 < t \leq 0$ and $g(t) \geq g(0) =0$ if $t \leq 0$. Thus, $\theta_1(\varepsilon)$ and $\theta_2(\varepsilon)$ are established by the above formulas. Next we give a little more information about $\theta_1(\varepsilon)$. That is,
\[
\theta_1(\varepsilon) \geq \inf_{t>-1/2}\big\{t\log\beta + (a + \varepsilon)\mathrm{P}(t+1)\big\}.
\]
In fact, it is easy to check that $\theta_1(\varepsilon) \geq \inf_{t>0}\big\{t\log\beta + (a + \varepsilon)\mathrm{P}(t+1)\big\}$ since they are both negative. Moreover, the first remark has indicated that the infmum can take over all $t > -1/2$.
\end{remark}

The following is a key lemma in the proof of the inequality (\ref{he}).

\begin{lemma}\label{l-n xiaoyu}
Let $\beta >1$ be a real number and $i \geq 0$ be an integer. Then for any $n \geq 1$,
\[
\lambda \left\{x \in [0,1): l_n(x) \geq i\right\} \leq \frac{\beta^{1-i}}{\beta-1}.
\]
\end{lemma}

\begin{proof}
By the definition of $l_n(x)$ in (\ref{l-n}), it is clear to see that the result is true for $i=0$.
Now let $i \geq 1$ be an integer. Then the set $\{x \in [0,1): l_n(x) \geq i\}$ is the union of the $(n+i)$-th cylinders $J(\varepsilon_1,\cdots,\varepsilon_n, \underbrace{0,\cdots,0}_{i})$, where $(\varepsilon_1,\cdots,\varepsilon_n) \in \Sigma_{\beta}^n$. That is,
\[
 \{x \in [0,1): l_n(x) \geq i\} = \bigcup_{(\varepsilon_1,\cdots,\varepsilon_n) \in \Sigma_{\beta}^n} J(\varepsilon_1,\cdots,\varepsilon_n, \underbrace{0,\cdots,0}_{i}).
\]
Since the cylinders $J(\varepsilon_1,\cdots,\varepsilon_n)$ and $J(\varepsilon^\prime_1,\cdots,\varepsilon^\prime_n)$ are disjoint for any $(\varepsilon_1,\cdots,\varepsilon_n)\neq(\varepsilon^\prime_1,\cdots,\varepsilon^\prime_n) \in \Sigma_{\beta}^n$ and the length of $J(\varepsilon_1,\cdots,\varepsilon_n, \underbrace{0,\cdots,0}_{i})$ always satisfies
$$|J(\varepsilon_1,\cdots,\varepsilon_n, \underbrace{0,\cdots,0}_{i})| \leq \frac{1}{\beta^{n+i}}$$ for any $(\varepsilon_1,\cdots,\varepsilon_n) \in \Sigma_{\beta}^n$, it is from Proposition \ref{Renyi} that
\begin{align*}
\lambda \left\{x \in [0,1): l_n(x) \geq i\right\} &= \sum_{(\varepsilon_1,\cdots,\varepsilon_n) \in \Sigma_{\beta}^n} |J(\varepsilon_1,\cdots,\varepsilon_n, \underbrace{0,\cdots,0}_{i})|\\
&\leq  \sharp \Sigma_{\beta}^n \cdot \frac{1}{\beta^{n+i}} \leq \frac{\beta^{n+1}}{\beta - 1} \cdot \frac{1}{\beta^{n+i}} = \frac{\beta^{1-i}}{\beta-1}.
\end{align*}
\end{proof}

Now we are going to give the proof of (\ref{he}).

\begin{proof}[Proof of (\ref{he})]
For any $x \in [0,1)$ and $n \geq 1$, Proposition \ref{zhongyao1} shows that
 \begin{equation*}
|J(\varepsilon_1(x) ,\varepsilon_2(x), \cdots, \varepsilon_n(x))| \geq \frac{1}{\beta^{n+l_n(x)+1}},
\end{equation*}
where $l_n(x)$ is defined as (\ref{l-n}). Let $m \geq 1$ be an integer. In view of (\ref{cylinder inequality}) and Lemma \ref{jihe}, we deduce that
\begin{equation}\label{dayu1}
\left\{x \in \mathbb{I}: k_n(x) \geq m\right\} \subseteq \left\{x \in \mathbb{I}: \frac{1}{\beta^{n+l_n(x)+1}} \leq \frac{1}{q_m^2(x)}\right\}.
\end{equation}
Now we claim that for any $\delta >0$,
\[
\left\{x \in \mathbb{I}: \frac{1}{\beta^{l_n(x)}} \leq \frac{\beta^{n+1}}{q_m^2(x)}\right\}
\subseteq \left\{x \in \mathbb{I}: \frac{1}{\beta^{l_n(x)}} \leq \delta\right\}
\bigcup \left\{x \in \mathbb{I}: \frac{\beta^{n+1}}{q_m^2(x)} \geq \delta\right\}.
\]
Otherwise, if there exists some real number $\delta_0 > 0$ such that $1/(\beta^{l_n(x)}) > \delta_0$ and $\beta^{n+1}/(q_m^2(x)) < \delta_0~(x \in \mathbb{I})$, then we have $1/(\beta^{l_n(x)}) > \beta^{n+1}/(q_m^2(x))$. Combing this with (\ref{dayu1}), we obtain that
\begin{equation}\label{dayu2}
\lambda \big\{x \in \mathbb{I}:k_n(x) \geq m\big\} \leq \lambda \left\{x \in \mathbb{I}: \frac{1}{\beta^{l_n(x)}}\leq \delta\right\} + \lambda \left\{x \in \mathbb{I}:q_m^{-2}(x) \geq \frac{\delta}{\beta^{n+1}}\right\}.
\end{equation}
Lemma \ref{l-n xiaoyu} implies that
\[
\lambda \left\{x \in \mathbb{I}: \frac{1}{\beta^{l_n(x)}}\leq \delta\right\} = \lambda\left\{x \in \mathbb{I}: l_n(x) \geq -\log_\beta\delta \right\} \leq C_\beta \delta,
\]
where $\log_\beta$ denotes the logarithm w.r.t.~the base $\beta$ and $C_\beta = \beta^2/(\beta -1)$. For any $t>0$, the Markov's inequality shows that
\begin{align*}
\lambda \left\{x \in \mathbb{I}:q_m^{-2}(x) \geq \frac{\delta}{\beta^{n+1}}\right\} =
&\lambda \left\{x \in \mathbb{I}:q_m^{-2t}(x) \geq \left(\frac{\delta}{\beta^{n+1}}\right)^t\right\}
\leq \frac{\beta^{t(n+1)}\mathrm{E}\left(q_m^{-2t}\right)}{\delta^t}.
\end{align*}
Combining this with (\ref{dayu2}), we have that
\begin{equation*}
\lambda \big\{x \in \mathbb{I}:k_n(x) \geq m\big\} \leq C_\beta \delta + \frac{\beta^{t(n+1)}\mathrm{E}\left(q_m^{-2t}\right)}{\delta^t}.
\end{equation*}
It follows form Lemma \ref{proofs} that
\[
\mathrm{P}(t+1) = \lim_{m \to \infty}\frac{1}{m} \log \mathrm{E}\left(q_m^{-2t}\right).
\]
Hence, for any $\eta >0$, there exists a positive number $M$ (depending on $\eta $) such that for all $m \geq M$, we have
\[
\mathrm{E}\left(q_m^{-2t}\right) \leq e^{m\left(\mathrm{P}(t+1) + \eta\right)}.
\]
Therefore, for any $m \geq M$, we obtain that
\begin{equation}\label{elementary}
\lambda \big\{x \in \mathbb{I}:k_n(x) \geq m\big\} \leq C_\beta \delta + \frac{\beta^{t(n+1)}e^{m\left(\mathrm{P}(t+1) + \eta\right)}}{\delta^t}.
\end{equation}
For any $\varepsilon> 0$ and $n \geq 1$, let $m_n = [n(a+\varepsilon)]$. Then $m_n \to \infty$,
$m_n/n \leq a+\varepsilon$ and $m_n/n \to a+\varepsilon$ as $n \to \infty$. So, there exists a positive number $N$ (depending on $\eta $ and $\varepsilon$) such that for all $n \geq N$, we have that $m_n \geq M$. Fixed such $n \geq N$, it follows from (\ref{elementary}) that
\begin{equation*}
\lambda \big\{x \in \mathbb{I}:k_n(x) \geq m_n\big\} \leq C_\beta \delta + \frac{\beta^{t(n+1)}e^{m_n\left(\mathrm{P}(t+1) + \eta\right)}}{\delta^t}.
\end{equation*}
Now we choose a suitable $\delta > 0$ such that $f(\delta) =  C_\beta \delta +\delta^{-t} \beta^{t(n+1)}e^{m_n\left(\mathrm{P}(t+1) + \eta\right)}$ reaches the minimum value. To do this, letting the derivative of $f(\delta)$ equals to zero, we calculate that
$\delta = \left(C_\beta^{-1}t \beta^{t(n+1)}e^{m_n\left(\mathrm{P}(t+1) + \eta\right)}\right)^{\frac{1}{t+1}}$.
Thus we deduce that
\begin{equation}\label{dayu3}
\lambda \big\{x \in \mathbb{I}:k_n(x) \geq m_n\big\} \leq H(t,\beta) \cdot \left(\beta^{t(n+1)}e^{m_n\left(\mathrm{P}(t+1) + \eta\right)}\right) ^{\frac{1}{t+1}},
\end{equation}
where $H(t,\beta) = \left(C_\beta^t \cdot t\right)^{\frac{1}{t+1}}(1+t^{-1})$ is a constant only depending on $t$ and $\beta$.
Since $m_n/n \leq a+\varepsilon$, we obtain that
\begin{equation*}
\lambda \left\{x \in \mathbb{I}:\frac{k_n(x)}{n} \geq a + \varepsilon \right\}
\leq \lambda \big\{x \in \mathbb{I}:k_n(x) \geq m_n\big\}.
\end{equation*}
Combing this with (\ref{dayu3}), we have that
\[
\lambda \left\{x \in \mathbb{I}:\frac{k_n(x)}{n} \geq a + \varepsilon \right\}  \leq H(t,\beta) \cdot \left(\beta^{t(n+1)}e^{m_n\left(\mathrm{P}(t+1) + \eta\right)}\right) ^{\frac{1}{t+1}}.
\]
Notice that $m_n/n \to a+\varepsilon$ as $n \to \infty$, we know that
\begin{equation*}
\limsup_{n \to \infty} \frac{1}{n}\log \lambda\left\{x \in \mathbb{I}:\frac{k_n(x)}{n} \geq a + \varepsilon \right\}  \leq \frac{1}{t+1}\Big(t\log\beta + (a + \varepsilon)\left(\mathrm{P}(t+1) + \eta\right))\Big)
\end{equation*}
and hence that
\[
\limsup_{n \to \infty} \frac{1}{n}\log \lambda\left\{x \in \mathbb{I}:\frac{k_n(x)}{n} \geq a + \varepsilon \right\}  \leq \frac{1}{t+1}\Big(t\log\beta + (a + \varepsilon)\mathrm{P}(t+1)\Big)
\]
holds for any $t >0$ since $\eta >0$ is arbitrary.

Therefore,
\begin{equation*}
\limsup_{n \to \infty} \frac{1}{n}\log \lambda\left\{x \in \mathbb{I}:\frac{k_n(x)}{n} \geq a + \varepsilon \right\}  \leq \theta_1(\varepsilon)
\end{equation*}
with
\[
\theta_1(\varepsilon) = \inf_{t>0} \left\{\frac{1}{t+1}\Big(t\log\beta + (a + \varepsilon)\mathrm{P}(t+1)\Big)\right\}.
\]
Now it remains to prove that $\theta_1(\varepsilon) <0$. Let $h(\omega)$ be the function defined as
\[
h(\omega) = \omega\log\beta + (a + \varepsilon)\mathrm{P}(\omega+1)\ \ \  \text{for all}\ \omega > -1/2.
\]
It is clear to see $h$ is real-analytic on $(-1/2, +\infty)$ since the pressure function $\mathrm{P}(\cdot)$ is real-analytic. By the properties of $\mathrm{P}(\cdot)$ in (\ref{ele}), we know that $h(0) = 0$ and $h^{\prime}(0) = -\pi^2\varepsilon/(6\log2) < 0$. Hence there exists $t_0 >0$ such that $h(t_0) < 0$ by the definition of derivative. Thus, $\theta_1(\varepsilon) \leq h(t_0) <0$.
\end{proof}

To prove the inequality (\ref{cha}), we need the following lemma whose proof is inspired by Wu \cite{lesWu06} (see also Li and Wu \cite{lesL.W08}).

\begin{lemma}\label{jihe2}
Let $n \geq 1$, $i \geq 1$ be integers and $x \in [0,1)$ be an irrational number such that $J(\varepsilon_1(x), \cdots, \varepsilon_n(x)) \not\subseteq I(a_1(x), \cdots, a_{i+1}(x))$. Then
\[
\frac{1}{6q_{i+3}^2(x)}\leq |J(\varepsilon_1(x), \cdots, \varepsilon_n(x))| \leq \frac{1}{\beta^n}.
\]
\end{lemma}

\begin{proof}
Since $J(\varepsilon_1(x), \cdots, \varepsilon_n(x)) \not\subseteq I(a_1(x), \cdots, a_{i+1}(x))$, we know that at least one endpoint of $J(\varepsilon_1(x), \cdots, \varepsilon_n(x))$ does not belong to $I(a_1(x), \cdots, a_{i+1}(x))$. Without loss of generality, we assume that the right endpoint of $J(\varepsilon_1(x), \cdots, \varepsilon_n(x))$ does not belong to $I(a_1(x), \cdots, a_{i+1}(x))$, i.e., the right endpoint of $I(a_1(x), \cdots, a_{i+1}(x))$ belongs to $J(\varepsilon_1(x), \cdots, \varepsilon_n(x))$.\\
{\bf Case I}. If $i$ is even, we know that $I(a_1(x), \cdots, a_{i+1}(x))$ is decomposed into a countable $(i+2)$-th cylinders like $I(a_1(x), \cdots, a_{i+1}(x), j)~(j \in \mathbb{N}$) and these cylinders $I(a_1(x), \cdots, a_{i+1}(x), 1), I(a_1(x), \cdots, a_{i+1}(x), 2), \cdots$ run from left to right. Since
$x \in J(\varepsilon_1(x), \cdots, \varepsilon_n(x)) \bigcap I(a_1(x), \cdots, a_{i+1}(x), a_{i+2}(x))$, we have that
\[
I(a_1(x), \cdots, a_{i+1}(x), a_{i+2}(x)+1) \subseteq J(\varepsilon_1(x), \cdots, \varepsilon_n(x)).
\]
By (\ref{recursive}) and (\ref{cylinder inequality}), we deduce that
\[
|I(a_1(x), \cdots, a_{i+1}(x), a_{i+2}(x)+1)| \geq \frac{1}{6q_{i+2}^2(x)} \geq \frac{1}{6q_{i+3}^2(x)}.
\]
Hence, in view of Proposition \ref{recursive}, we obtain
\[
\frac{1}{6q_{i+3}^2(x)} \leq  |J(\varepsilon_1(x), \cdots, \varepsilon_n(x))| \leq \frac{1}{\beta^n}.
\]
{\bf Case II}. If $i$ is odd, we consider the $(i+2)$-th cylinder $I(a_1(x), \cdots, a_{i+1}(x), a_{i+2}(x))$. We know that $I(a_1(x), \cdots, a_{i+1}(x), a_{i+2}(x))$ can be decomposed into a countable $(i+3)$-th cylinders like $I(a_1(x), \cdots, a_{i+1}(x), a_{i+2}(x), j)~(j \in \mathbb{N}$) and these cylinders $I(a_1(x), \cdots, a_{i+1}(x), a_{i+2}(x), 1), I(a_1(x), \cdots, a_{i+1}(x), a_{i+2}(x), 2), \cdots$ also run from left to right. Notice that $x \in J(\varepsilon_1(x), \cdots, \varepsilon_n(x)) \bigcap I(a_1(x), \cdots, a_{i+2}(x), a_{i+3}(x))$, we obtain that $I(a_1(x), \cdots, a_{i+1}(x), a_{i+3}(x)+1) \subseteq J(\varepsilon_1(x), \cdots, \varepsilon_n(x))$. Therefore,
\[
|J(\varepsilon_1(x), \cdots, \varepsilon_n(x))| \geq \frac{1}{6q_{i+3}^2(x)}.
\]
By Proposition \ref{recursive}, we complete the proof.
\end{proof}

Now we are ready to prove (\ref{cha}).

\begin{proof}[Proof of (\ref{cha})]
Let $m \geq 1$ be an integer. By Lemma \ref{jihe}, we obtain that
\[
\left\{x \in \mathbb{I}: k_n(x) \leq m\right\} = \left\{x \in \mathbb{I}: J(\varepsilon_1(x), \cdots, \varepsilon_n(x)) \not\subseteq I(a_1(x), \cdots, a_{m+1}(x))\right\}.
\]
It follows from Lemma \ref{jihe2} that
\[
\left\{x \in \mathbb{I}: J(\varepsilon_1(x), \cdots, \varepsilon_n(x)) \not\subseteq I(a_1(x), \cdots, a_{m+1}(x))\right\} \subseteq \left\{x \in \mathbb{I}: \frac{1}{6q_{m+3}^2(x)} \leq \frac{1}{\beta^n}\right\}.
\]
Therefore,
\[
\lambda\left\{x \in \mathbb{I}: k_n(x) \leq m\right\} \leq \lambda \left\{x \in \mathbb{I}:q_{m+3}^2(x) \geq \frac{\beta^n}{6}\right\}.
\]
For any $0 < t <1/2$, the Markov's inequality yields that
\begin{equation}\label{xiaoyu1}
\lambda\left\{x \in \mathbb{I}: k_n(x) \leq m\right\}
 \leq \lambda \left\{x \in \mathbb{I}: q_{m+3}^{2t}(x) \geq \left(\frac{\beta^n}{6}\right)^t\right\}
 \leq \frac{6^t \mathrm{E}\left(q_{m+3}^{2t}\right)}{\beta^{tn}}.
\end{equation}
By Lemma \ref{ele}, for any $\eta >0$, there exists a positive number $M$ (depending on $\eta$) such that for all $m \geq M$, we have
\[
\mathrm{E}\left(q_{m+3}^{2t}\right) \leq e^{m(\mathrm{P}(1-t)+\eta)}.
\]
Combing this with (\ref{xiaoyu1}), for any $m \geq M$, we obtain that
\begin{equation}\label{xiaoyu2}
\lambda\left\{x \in \mathbb{I}: k_n(x) \leq m\right\} \leq \frac{6^t e^{m(\mathrm{P}(1-t)+\eta)}}{\beta^{tn}}.
\end{equation}
For any $0<\varepsilon<a$ and $n \geq 1$, let $m_n = [n(a-\varepsilon)]+1$. Then $m_n \to \infty$, $m_n/n \geq a-\varepsilon$ and $m_n/n \to a-\varepsilon$ as $n \to \infty$. So, there exists a positive number $N$ (depending on $\eta$ and $\varepsilon$) such that for all $n \geq N$, we have that $m_n \geq M$. Now we fix such $n \geq N$, in view of (\ref{xiaoyu2}), we deduce that
\begin{equation*}
\lambda\left\{x \in \mathbb{I}: k_n(x) \leq m_n\right\} \leq 6^t\beta^{-tn} e^{m_n(\mathrm{P}(1-t)+\eta)}.
\end{equation*}
Notice that $m_n/n \geq a-\varepsilon$, we have
\begin{equation*}
\lambda\left\{x \in \mathbb{I}: \frac{k_n(x)}{n}\leq a - \varepsilon \right\} \leq \lambda\left\{x \in \mathbb{I}: k_n(x) \leq m_n\right\} \leq
6^t\beta^{-tn} e^{m_n(\mathrm{P}(1-t)+\eta)}
\end{equation*}
and hence that
\[
\limsup_{n \to \infty} \frac{1}{n}\log \lambda\left\{x \in \mathbb{I}: \frac{k_n(x)}{n}\leq a - \varepsilon \right\} \leq -t\log\beta + (a - \varepsilon) \mathrm{P}(1-t)
\]
for any $0<t < 1/2$ since $m_n/n \to a-\varepsilon$ as $n \to \infty$ and $\eta >0$ is arbitrary.
Therefore,
\begin{equation*}
\limsup_{n \to \infty} \frac{1}{n}\log \lambda\left(\frac{k_n}{n} \leq a - \varepsilon\right) \leq \theta_2(\varepsilon)
\end{equation*}
with
\[
\theta_2(\varepsilon) = \inf_{0 <t< 1/2} \big\{-t\log\beta + (a - \varepsilon)\mathrm{P}(1-t)\big\}.
\]
Now we  need to show that $\theta_2(\varepsilon) < 0$. For any $\omega <1/2$, we consider the function
\[
h(\omega) = -\omega\log\beta + (a - \varepsilon)\log\mathrm{P}(1-\omega).
\]
Then it is easy to check that $h$ is real-analytic on $(-\infty, 1/2)$ and satisfies $h(0) = 0$ and $h^{\prime}(0) = -\pi^2\varepsilon/(6\log2) < 0$ because of the properties of $\mathrm{P}(\cdot)$ in (\ref{ele}). So, if $t>0$ sufficiently close to $0$, we obtain that $\theta_2(\varepsilon) < 0$.
\end{proof}

We end this section with the proof of Theorem \ref{error term}.

\begin{proof}[Proof of Theorem \ref{error term}]
For any $\varepsilon>0$ and $n \geq 1$, since
\begin{align*}
 &\lambda \left\{x \in \mathbb{I}: \left|\frac{k_n(x)}{n} - a\right| \geq \varepsilon\right\} \notag \\
=\ &\lambda \left\{x \in \mathbb{I}: \frac{k_n(x)}{n} \geq a + \varepsilon\right\} + \lambda \left\{x \in \mathbb{I}: \frac{k_n(x)}{n} \leq a - \varepsilon\right\},
\end{align*}
we obtain that
\begin{align*}
 &\limsup_{n \to \infty} \frac{1}{n}\log \lambda \left\{x \in \mathbb{I}: \left|\frac{k_n(x)}{n} - a\right| \geq \varepsilon\right\} \\
=&\limsup_{n \to \infty} \frac{1}{n}\log \left(\lambda \left\{x \in \mathbb{I}: \frac{k_n(x)}{n} \geq a + \varepsilon\right\} + \lambda \left\{x \in \mathbb{I}: \frac{k_n(x)}{n} \leq a - \varepsilon\right\}\right)\\
\leq& \max\{\theta_1(\varepsilon),\theta_2(\varepsilon)\},
\end{align*}
where the last inequality follows from the inequalities (\ref{he}) and (\ref{cha}) and $\theta_1(\varepsilon)$ and $\theta_2(\varepsilon)$ are as defined in Proposition \ref{first}. Therefore, for any $\varepsilon>0$, there exist positive real $\alpha$ (only depending on $\beta$ and $\varepsilon$) and positive integer $N$ such that for all $n > N$, we have
\begin{equation}\label{probability}
\lambda \left\{x \in \mathbb{I}: \left|\frac{k_n(x)}{n} - a\right| \geq \varepsilon\right\} \leq e^{-\alpha n}.
\end{equation}
For all $1 \leq n \leq N$, since the probabilities of the left-hand side in (\ref{probability}) are bounded, we can choose sufficiently large $A$ (only depending on $\beta$ and $\varepsilon$) such that
\[
\lambda \left\{x \in \mathbb{I}: \left|\frac{k_n(x)}{n} - a\right| \geq \varepsilon\right\} \leq Ae^{-\alpha n}
\]
holds for all $n \geq 1$. Thus, we complete the proof of Theorem \ref{error term}.

\end{proof}

\subsection{Proof of Theorem \ref{better approximation}}

Being similar to the proof of Theorem \ref{error term}, we give a stronger result than that of Theorem \ref{better approximation} as well.

\begin{proposition}\label{second}
Let $\beta >1$ be a real number.\\
(i) If $\log\beta > {\pi^2}/{(6\log2)}$, then
\begin{equation}\label{better 1}
\limsup_{n \to \infty} \frac{1}{n}\log \lambda \left\{x \in \mathbb{I}: \left|x - \frac{p_n}{q_n}\right| \leq x -x_n \right\} \leq \theta
\end{equation}
with
\[
\theta = \inf_{0< t< 1/2} \big\{-t\log\beta + \mathrm{P}(1-t)\big\} <0.
\]
(ii) If $\log\beta < {\pi^2}/{(6\log2)}$, then
\begin{equation}\label{better 2}
\limsup_{n \to \infty} \frac{1}{n}\log \lambda \left\{x \in \mathbb{I}: x -x_n \leq \left|x - \frac{p_n}{q_n}\right|\right\} \leq \theta^{\ast}
\end{equation}
with
\[
\theta^{\ast} = \inf_{t>0} \left\{\frac{1}{t+1}\Big(t\log\beta + \mathrm{P}(t+1)\Big)\right\} <0.
\]
\end{proposition}

\begin{remark}\label{3.6}
By the similar methods with Remark \ref{add remarks}, the constants $\theta$ and $\theta^{\ast}$ can also be rewritten as
\[
\theta = \inf_{t< 1/2} \big\{-t\log\beta + \mathrm{P}(1-t)\big\}
\]
and
\[
\theta^{\ast} = \inf_{t>-1/2} \Big\{\frac{1}{t+1}\left(t\log\beta + \mathrm{P}(t+1)\Big)\right\}.
\]
Besides, we can also give more remarks on $\theta$ and $\theta^{\ast}$, which indicates that $\theta$ and $\theta^{\ast}$ are related to the multifractal analysis for the Lyapunov exponent
of the Gauss transformation. Recall that Kesseb\"{o}hmer and Stratmann \cite[Theorem 1.3]{lesKS07} (see also Fan et al. \cite{lesF.L.W.W}) proved that
\[
\tau(\gamma):= \dim_H \left\{x \in [0,1): \mathcal{L}(x) = \gamma \right\} = \frac{\inf_{t \in \mathbb{R}}\{t\cdot \gamma + \mathrm{P}(t)\}}{\gamma}
\]
for any $\gamma \geq 2\log((\sqrt{5}+1)/2)$, where $\dim_H$ denotes the Hausdorff dimension and $\mathcal{L}(x)$ is the Lyapunov exponent of the Gauss transformation $T$ defined as
\[
\mathcal{L}(x):= \lim\limits_{n \to \infty}\frac{1}{n}\log |(T^n)^{\prime}(x)|,
\]
if the limit exists. Since $\log\beta > {\pi^2}/{(6\log2)} > 2\log((\sqrt{5}+1)/2)$, a simple calculation implies that the constant $\theta$ have an alternative form
\[
\theta = (\tau(\log \beta) -1)\log \beta.
\]
It is not difficult to check that
\[
\theta^{\ast} \geq \inf_{t>-1/2} \big\{t\log\beta + \mathrm{P}(t+1)\big\} = (\tau(\log \beta) -1)\log \beta,
\]
where the last equality only holds for $\log\beta \geq 2\log((\sqrt{5}+1)/2)$.
By Theorem 1.3 of Kesseb\"{o}hmer and Stratmann \cite{lesKS07} (see also Theorem 1.3 of Fan et al.~\cite{lesF.L.W.W}), we know that $-(\log \beta)/2 < \theta <0$ for $\log\beta > {\pi^2}/{(6\log2)}$ and $\theta^{\ast} \geq -\log\beta$ for $\log\beta \geq 2\log((\sqrt{5}+1)/2)$. This is one way to show $\theta$ is negative and also gives the lower bounds for $\theta$ and $\theta^{\ast}$.
\end{remark}

We first give the proof of the inequality (\ref{better 1}).

\begin{proof}[Proof of (\ref{better 1})]
For any irrational number $x \in [0,1)$ and $n \geq 1$, in view of (\ref{beta expansion}) and (\ref{Diophantine}), we have that
\[
x - x_n =\frac{T_{\beta}^nx}{\beta^n} \leq \frac{1}{\beta^n}\ \ \ \  \text{and}\ \ \ \ \left|x -\frac{p_n(x)}{q_n(x)}\right| \geq  \frac{1}{2q_{n+1}^{2}(x)}.
\]
Hence, we obtain that
\[
\lambda \left\{x \in \mathbb{I}: \left|x - \frac{p_n}{q_n}\right| \leq x -x_n \right\} \leq \lambda \left\{x \in \mathbb{I}:  \frac{1}{2q_{n+1}^2(x)} \leq \frac{1}{\beta^n}\right\}.
\]
For any $0 < t <1/2$, the Markov's inequality implies that
\[
\lambda \left\{x \in \mathbb{I}:  \frac{1}{2q_{n+1}^2(x)} \leq \frac{1}{\beta^n}\right\}
\leq \lambda \left\{x \in \mathbb{I}: q_{n+1}^{2t}(x) \geq \left(\frac{\beta^{n}}{2}\right)^t\right\}
\leq \frac{2^t \mathrm{E}(q_{n+1}^{2t})}{\beta^{tn}}.
\]
The similar methods of (\ref{xiaoyu2}) yield that
\begin{equation*}
\limsup_{n \to \infty} \frac{1}{n}\log \lambda \left\{x \in \mathbb{I}: \left|x - \frac{p_n}{q_n}\right| \leq x -x_n \right\} \leq -t\log\beta + \mathrm{P}(1-t)
\end{equation*}
for any $0 < t <1/2$. Therefore,
\begin{equation*}
\limsup_{n \to \infty} \frac{1}{n}\log \lambda \left\{x \in \mathbb{I}: \left|x - \frac{p_n}{q_n}\right| \leq x -x_n \right\} \leq \theta
\end{equation*}
with
\[
\theta = \inf_{0< t< 1/2} \big\{-t\log\beta + \mathrm{P}(1-t)\big\}.
\]
The condition $\log\beta > \pi^2/(6\log2)$ assures that $\theta <0$ using the similar techniques at the end of the proof of (\ref{cha}).
\end{proof}

Next we prove the inequality (\ref{better 2}).

\begin{proof}[Proof of (\ref{better 2})]
For any irrational number $x \in [0,1)$ and $n \geq 1$, by (\ref{beta expansion}) and (\ref{Diophantine}), we know that
\[
x - x_n =\frac{T_{\beta}^nx}{\beta^n} \ \ \ \  \text{and}\ \ \ \ \left|x -\frac{p_n(x)}{q_n(x)}\right| \leq \frac{1}{q_n^{2}(x)}.
\]
Thus
\begin{equation}\label{1}
\left\{x \in \mathbb{I}: x -x_n \leq \left|x - \frac{p_n}{q_n}\right|\right\} \leq \left\{x \in \mathbb{I}: \frac{T_{\beta}^nx}{\beta^n} \leq \frac{1}{q_n^2(x)}\right\}.
\end{equation}
Note that for any $\delta >0$,
\[
\left\{x \in \mathbb{I}: \frac{T_{\beta}^nx}{\beta^n} \leq \frac{1}{q_n^2(x)}\right\} \subseteq \big\{x \in \mathbb{I}: T_{\beta}^nx \leq \delta\big\} \bigcup \left\{x \in \mathbb{I}: \frac{\beta^n}{q_n^2(x)} \geq \delta\right\},
\]
otherwise, if there exists some real number $\delta_0 > 0$ such that $T_{\beta}^nx > \delta_0$ and $\beta^{n}/(q_n^2(x)) < \delta_0~(x \in \mathbb{I})$, then we have $T_{\beta}^nx  > \beta^{n}/(q_n^2(x))$ and hence that $T_{\beta}^nx /\beta^n > 1/(q_n^2(x))$. Therefore,
\begin{equation}\label{2}
\lambda \left\{x \in \mathbb{I}: \frac{T_{\beta}^nx}{\beta^n} \leq \frac{1}{q_n^2(x)}\right\}
\leq \lambda \big\{x \in \mathbb{I}: T_{\beta}^nx \leq \delta\big\} +
\lambda \left\{x \in \mathbb{I}: q_n^{-2}(x) \geq \frac{\delta}{\beta^{n}}\right\}.
\end{equation}
Since $T_{\beta}$ is measure-preserving w.r.t.~$\mu$, we have $\mu\{x \in \mathbb{I}: T_{\beta}^nx \leq \delta\} = \delta$. The relation (\ref{beta measure}) between $\mu$ and $\mathrm{P}$ yields that
\[
\lambda\big\{x \in \mathbb{I}: T_{\beta}^nx \leq \delta\big\} \leq C \delta,
\]
where $C >1$ is a constant only depending on $\beta$. For any $t>0$, the Markov's inequality indicates that
\[
\lambda \left\{x \in \mathbb{I}: q_n^{-2}(x) \geq \frac{\delta}{\beta^{n}}\right\}
= \lambda \left\{x \in \mathbb{I}: q_n^{-2t}(x)\geq \left(\frac{\delta}{\beta^n}\right)^t\right\} \leq
\frac{ \beta^{tn} \mathrm{E}(q_n^{-2t})}{\delta^t}.
\]
Combining these with (\ref{1}) and (\ref{2}), we deduce that
\[
\left\{x \in \mathbb{I}: x -x_n \leq \left|x - \frac{p_n}{q_n}\right|\right\} \leq C \delta + \frac{ \beta^{tn} \mathrm{E}(q_n^{-2t})}{\delta^t}.
\]
Using similar methods of the proof of (\ref{elementary}) and (\ref{dayu3}), we actually obtain that
\begin{equation*}
\limsup_{n \to \infty} \frac{1}{n}\log \lambda \left\{x \in \mathbb{I}: x -x_n \leq \left|x - \frac{p_n}{q_n}\right|\right\}  \leq \frac{1}{t+1}\Big(t\log\beta + \mathrm{P}(t+1)\Big)
\end{equation*}
for any $t>0$ and hence that
\begin{equation*}
\limsup_{n \to \infty} \frac{1}{n}\log \lambda \left\{x \in \mathbb{I}: x -x_n \leq \left|x - \frac{p_n}{q_n}\right|\right\} \leq \theta^{\ast}
\end{equation*}
with
\[
\theta^{\ast} = \inf_{t>0} \Big\{\frac{1}{t+1}\left(t\log\beta + \mathrm{P}(t+1)\Big)\right\}.
\]
The condition $\log\beta < \pi^2/(6\log2)$ guarantees that $\theta^{\ast} <0$ by the similar techniques at the end of the proof of (\ref{he}).
\end{proof}

At last, we are ready to prove Theorem \ref{better approximation}.

\begin{proof}[Proof of  Theorem \ref{better approximation}]
In view of Proposition \ref{second}, the similar methods of the proof of Theorem \ref{error term} give the proofs of (i) and (ii) in Theorem \ref{better approximation}.
\end{proof}

{\bf Acknowledgement}
The work was supported by NSFC 11371148, 11201155, 11271140 and Guangdong Natural Science Foundation 2014A030313230.

\end{document}